\documentclass
{amsart}

\usepackage{amssymb,amscd}
\usepackage[all,line,arc,curve,color,frame,pdf]{xy}
\usepackage{tikz} 
\usepackage{tikz-cd}
\usepackage{xcolor}
\usetikzlibrary{positioning, trees, decorations.pathmorphing}
\usepackage[textsize=tiny]{todonotes}
\usepackage{hyperref}
\usepackage[shortlabels]{enumitem}
\usepackage{float} 
\usepackage[paper=a4paper, margin=3.5cm]{geometry}
\usepackage{cleveref}  
\usepackage{comment}
\usepackage{scalerel}
\usepackage[normalem]{ulem}
\usepackage{booktabs}
\usepackage{mathtools}
\usepackage{cite} 
\usepackage{orcidlink}

\usepackage{subcaption}

\numberwithin{equation}{section}
\theoremstyle{plain}
\newtheorem{thm}{Theorem}[section]
\newtheorem{cor}[thm]{Corollary}
\newtheorem{lemma}[thm]{Lemma}
\newtheorem{prop}[thm]{Proposition}

\newtheorem{question}[thm]{Question}

\theoremstyle{remark}
\newtheorem{rem}[thm]{Remark}
\newtheorem{ex}[thm]{Example}
\newtheorem{conj}[thm]{Conjecture}

\theoremstyle{definition}
\newtheorem{defi}[thm]{Definition}

\def\NN{\mathbb{N}}
\def\ZZ{\mathbb{Z}}

\DeclareMathOperator{\Psd}{Psd}

\newcommand{\stirling}[2]{\left[\begin{smallmatrix} #1 \\ #2 \end{smallmatrix}\right]}

\newcommand*\dd{\mathop{}\!\mathrm{d}}
\DeclareMathOperator\CM{CM}
\DeclareMathOperator\disc{disc}
\DeclareMathOperator{\id}{id}

\def\L{\mathcal{L}}
\def\C{\mathcal{C}}

\def\RR{\mathbb{R}}

\def\NN{\mathbb{N}}

\def\F{\mathcal{F}}

\def\G{\mathcal{G}}
\def\B{\mathcal{B}}

\begin{document}
\begin{flushright}
    \texttt{USTC-ICTS/PCFT-25-16\\
    MPP-2025-77}
\end{flushright}

\title{Complete monotonicity of log-functions}

\begin{abstract}
    In this article we investigate the property of complete monotonicity within a special family $\mathcal{F}_s$ of  functions in $s$ variables involving logarithms. The main result of this work provides a linear isomorphism between $\mathcal{F}_s$ and the space of real multivariate polynomials.  This isomorphism identifies the cone of completely monotone functions with the cone of non-negative polynomials. We conclude that the cone of completely monotone functions in $\F_s$ is semi-algebraic. This gives a finite time algorithm to decide whether a function in $\F_s$ is completely monotone.
\end{abstract}
\thanks{$^a$ Interdisciplinary Center for Theoretical Study, University of Science and Technology of China, Hefei, Anhui 230026, China \\ Max-Planck-Institut f\"ur Physik,  Werner-Heisenberg-Institut, Boltzmannstraße 8, 85748 Garching, Germany. \\  Email address: marr21@mail.ustc.edu.cn}
\thanks{$^b$ Max Planck Institute for Mathematics in the Sciences/University of Leipzig, Leipzig, Germany. \\ Email address: julian.weigert@mis.mpg.de}
\thanks{ RM is supported by the Outstanding PhD Students
Overseas Study Support Program of University of Science and Technology of China}
\author[Rourou Ma]{Rourou Ma \orcidlink{0000-0001-5182-0648} $^a$}

\author[Julian Weigert]{Julian Weigert \orcidlink{https://orcid.org/0009-0001-3576-8699} $^b$}
\thanks{JW is supported by the SPP 2458 “Combinatorial Synergies”, funded
by the Deutsche Forschungsgemeinschaft (DFG, German Research Foundation),
project ID: 539677510.}

\maketitle

\section{Introduction}
Consider an infinitely differentiable function $f:\RR_{>0}^s\to \RR$. The function $f$ is said to be \emph{completely monotone (on the positive orthant)} if \begin{align}
\label{eq:cmineq}
   \forall\alpha\in \NN_0^s:\forall x\in \RR_{>0}^s: \prod_{i=1}^s\left(\frac{-\partial }{\partial x_i}\right)^{\alpha_i}f(x_1,\ldots,x_s) \geq 0 
\end{align}

Given an explicit function $f$, checking whether it satisfies this definition is an infinite task in two ways: There are both infinitely many derivatives and infinitely many points $x \in \RR_{>0}$ at which we need to check non-negativity. This means that the above conditions as written here can not be checked directly by a computer in finite time even for simple functions $f$. 

Denote by $\mathcal{C}^\infty(\RR_{>0}^s,\RR)$ the $\RR$-vector space of smooth functions from the positive orthant $\RR_{>0}^s$ to $\RR$. Let $\CM_s\subseteq \mathcal{C}^\infty(\RR_{>0}^s,\RR)$ be the subset of all completely monotone functions. Instead of checking complete monotonicity for a single function one can try to describe the set $\CM_s$. It is easy to see that $\CM_s$ forms a convex cone: For $f,g\in \CM_s, \lambda\in \RR_{>0}$ also $f+\lambda g\in \CM_s$. An overview on known methods in complete monotonicity can be found in \cite{Merkle2014}. We choose to further specialize the situation by restricting ourselves to a fixed subspace $\F_s\subseteq \C^\infty(\RR_{>0}^s,\RR)$. The following family of functions was suggested to us by Johannes Henn and Bernd Sturmfels due to its appearance in physics. 

 Due to the appearance of infrared and ultraviolet divergences in scattering amplitudes at some specific kinematic limits, multi-loop Feynman integral computations yield multi-valued functions in kinematic variables, like harmonic polylogarithms \cite{MAITRE2006222}. Log-functions appear as one of the simplest types in harmonic polylogarithms. From an analytic calculation point of view, the most popular method uses differential equations \cite{KOTIKOV1991158} at the level of Feynman integrals, based on integration by parts \cite{CHETYRKIN1981159}. While integration by parts provides us with a perspective to regard Feynman integrals as a linear space spanned by the irreducible integral basis, differential equations give us the function space information of the integral basis. The more properties of the integral basis we know, the more structure we may find in the amplitudes. For example, the discovery that the integral basis can be chosen as uniform transcendental weight integrals in \cite{Henn:2013pwa}, lead to the decomposition of the MHV partial amplitude into a Parke-Taylor factor and a factor of uniform transcendental integrals \cite{Arkani-Hamed:2014bca}.
Positivity is a property that gained popularity within the physics community in recent years. There could be a complete monotonicity property in the building blocks of scattering amplitudes in quantum field theory\cite{Henn:2024qwe}, such as two-loop Wilson loops\cite{Chicherin:2024hes}, QCD and QED cusp anomalous dimension\cite{Grozin_2015, Br_ser_2021}.
The space $\F_s$ is a function space that could appear in scattering amplitudes.

On the other hand, positivity has the potential to help performing the bootstrap method \cite{Almelid_2017}.
Bootstrap is a more efficient way to compute analytic amplitudes compared to the approach using differential equations. The drawback of bootstrap often is the lack of restrictive conditions. Positivity as a new property may provide additional restrictions in some cases \cite{simmonsduffin2016tasilecturesconformalbootstrap}.

The concept of complete monotonicity also plays an important role in several branches of mathematics such as probability theory \cite{ALZER201810,kimberling,OUIMET20181609}, real algebraic geometry \cite{realHyperbolic} and control theory \cite{controltheory}. The articles \cite{ALZER201810,OUIMET20181609} also use complete monotonicity as a tool for proving combinatorial inequalities.

\begin{defi}
    For $n\in \NN_0$ let $\RR[{\bf y}]_n:=\RR[y_1,\ldots,y_s]_n$ denote the $\RR$-vector space of all polynomials over $\RR$ in $s$ variables  of total degree at most $n$. We define \begin{align*}
\mathcal{F}_{s,n}:=\left\{f\in \mathcal{C}^\infty(\RR_{>0}^s,\RR) \middle| \exists p\in \RR[{\bf y}]_n: \forall x \in \RR_{>0}^s: f(x)=\frac{p(\log(x_1),\ldots, \log(x_s))}{x_1\cdots x_s}\right\}
    \end{align*}
We further set \begin{align*}
    \F_s:=\bigcup_{n\in \NN_0}\F_{s,n}.
\end{align*}
\end{defi}
The goal of this article is to describe the cone $\CM_s\cap \F_s$. In particular we aim to answer the following question
\begin{question}
     How can we test membership in the cone $\CM_s\cap \F_s$?
\end{question}

We provide an answer to this question by showing that for any $n\in \NN_0$ the cone $\CM_{s,n}:=\CM_s\cap \F_{s,n}$ is a semi-algebraic set: It can be described by finitely many polynomial inequalities in the coefficients of a function $f\in \F_{s,n}$.
More precisely we provide an explicit isomorphism $\F_{s,n}\simeq \RR[{\bf y}]_n$ which sends the cone $\CM_{s,n}$ to the well-known and much studied cone $\Psd_{s,n}$ of non-negative polynomials in $s$ variables of total degree at most $n$. For special values of $s,n$, including $s=1$, this cone coincides with the cone $\Sigma_{s,n}$ of all polynomials that can be written as sums of squares. In these cases membership can be tested efficiently, see \cite[Chapter 3]{RAG} for a good overview. 

We first describe the univariate case $s=1$.
Writing $\Gamma(z)=\int_0^\infty t^{z-1}\exp(-t)\dd t$ we set \begin{align*}
    g_l\coloneqq\left.\left(\frac{\dd }{\dd z}\right)^l\Gamma(z)\right|_{z=1}.
\end{align*}
Define a matrix $A=(A_{m,k})_{0\leq m ,k\leq n}$ by \begin{align*}
        A_{m,k}\coloneqq\begin{cases}
            (-1)^m\binom{k}{m}g_{k-m} &\text{if }k\geq m \\
            0 & \text{else}
        \end{cases}.
\end{align*}
By fixing the monomial basis $(1,y,y^2,\ldots,y^n)$ on $\RR[y]_n$ and the basis $x\mapsto \frac{\log(x)^j}{x},j=0,\ldots,n$, of $\F_{1,n}$ we view $A$ as a linear map $\varphi_A:\RR[y]_n\to \F_{1,n}$. More precisely this map sends the polynomial $\sum_{k=0}^nc_ky^k$ to the function $x\mapsto \sum_{k=0}^n\sum_{m=0}^nA_{m,k}c_k\frac{\log(x)^m}{x}$.
Our first result states that this is an isomorphism which maps the cone $\Psd_{1,n}$ isomorphically onto the cone $\CM_{1,n}$.
\begin{thm}
    \label{thm:intro_main_thm}
        Let $p=\sum_{i=0}^nc_iy^i\in \RR[y]_n$ and let $f:x \mapsto \frac{p(\log(x))}{x}$ be the corresponding function in $\F_n$. Write $c=(c_0,\ldots,c_n)^T$ and $Y=(1, y,\ldots,y^n)^T$, then $f$ is completely monotone if and only if \begin{align*}
            Y^TA^{-1}c\in \Psd_{1,n}=\Sigma_{1,n}.
        \end{align*}
\end{thm}

To move from univariate polynomials ($s=1$) to the general case we simply need to apply the linear map $\varphi_A$ to each variable seperately. We define a linear map\begin{align*}
    T^{s,n}\varphi_A: \RR[{\bf y}]_n\to \F_{s,n}
\end{align*}  
by declaring that the monomial $y_1^{\lambda_1}\cdots y_s^{\lambda_s}, \sum_{i=1}^s\lambda_i\leq n$ gets mapped to the function \begin{align*}
(x_1,\ldots,x_s)\mapsto \prod_{j=1}^s\varphi_A(y_j^{\lambda_j})(x_j)\in \F_{s,n}.
\end{align*}

\begin{thm}
    \label{thm:multi_main_thm_intro}
    A function $f\in \F_{s,n}$ is completely monotone if and only if $(T^{s,n}\varphi_A)^{-1}(f)\in \Psd_{s,n}$. In particular, $\CM_{s,n}$ is a semi-algebraic cone. 
\end{thm}

We also provide an explicit formula of the entries of $A^{-1}$ in terms of zeta-values. Apart from the fact that a computer can not store the exact numerical values appearing in $A^{-1}$, membership in $\Psd_{s,n}$ can be checked algorithmically in finite time. In the so-called Hilbert cases, which include the special case $s=1$, this is a semidefinite programming problem, see for instance \cite[Chapter 3.1.3]{RAG}.
Therefore the above theorem provides an effective tool for checking complete monotonicity within the family $\F_{s,n}$. 

For the proof of Theorem \ref{thm:intro_main_thm}  we use an important result by Bernstein, Hausdorff, Widder and Choquet that relates complete monotonicity of a function $f$ to non-negativity of its inverse Laplace transform. A second approach, which stays closer to the definition \ref{eq:cmineq} of complete monotonicity is highlighted in Section \ref{sec:infinite_order_derivative}. There we describe how the matrix $A$ and its inverse appear naturally when looking at higher derivatives of a function $f\in \F_{1,n}$. We provide a way of making sense of ``non-negativity of the infinite order derivative of $f$". In fact this condition of non-negativity is equivalent to the condition $Y^TA^{-1}c\in \Sigma_{1,n}$ in Theorem \ref{thm:intro_main_thm}. 

This article is structured as follows.
In Section \ref{sec:example_n=2} we present in detail the case $s=1,n=2$ which serves as an elementary example. We also give general formulas for the derivatives of functions in $\F_{1,n}$ for arbitrary $n$ as they generalize easily from the $n=2$ case. However, we will see that not all properties of the quadratic case readily hold in higher degrees. Section \ref{sec:main_result} contains the proof of our main results Theorem \ref{thm:intro_main_thm} and Theorem \ref{thm:multi_main_thm_intro}. Section \ref{sec:infinite_order_derivative} provides a different viewpoint on the univariate case in terms of higher order derivatives of a function in $\F_{1,n}$.


\section{The univariate quadratic case}
\label{sec:example_n=2}
We start by showcasing two different ways of computing the CM-region $\CM_{1,2}$. In the first approach we naively compute higher and higher derivatives to check the inequalities (\ref{eq:cmineq}). For the second approach we will invoke a powerful theorem due to Bernstein, Hausdorff, Widder and Choquet. 

For parameters $c_0,c_1,c_2$ let us consider \begin{align*}
    f:\RR_{>0}\to \RR, x \mapsto \frac{c_2\log(x)^2+c_1\log(x)+c_0}{x}.
\end{align*}
It is straightforward to see that the $k$-th derivative of $f$ satisfies 

\begin{align*}
        \left(\frac{\dd}{\dd x}\right)^kf(x)=\frac{h_k(\log(x))}{x^{k+1}}.
    \end{align*}
for some polynomials $h_k(y)\in \RR[y]_{2}$. By taking one more derivative, we observe that these polynomials must satisfy the recursion 
     \begin{align*}
        h_k(y)=\left(\frac{\dd}{\dd y}\right)h_{k-1}(y)-kh_{k-1}(y).
    \end{align*}
A similar recursion is satisfied by \emph{Stirling numbers of the first kind}, see \cite{Comtet}. We quickly disrupt our discussion of the $n=2$ case to make this observation more formal.
\begin{defi}
    For any $n,k \in \NN_0$, the associated \emph{signed Stirling number of the first kind} is denoted $c(n,k)$ and uniquely defined by 
    \begin{align} \nonumber 
    c(n,k)&=0 \ \ \text{ if } k>n \\ \nonumber
    c(n,0)&=\begin{cases}
        1 &\text{ if }n=0 \\
        0 &\text{ else}
    \end{cases}  \\
    \label{eq:recursionStirling}
    c(n,k)&=c(n-1,k-1)-(n-1)c(n-1,k) \ \ \text{ if } 0<k\leq n.
\end{align} 
We also define the unsigned version $\stirling{n}{k}:=|c(n,k)|$ which satisfies the recursion
\begin{align*}  
    \stirling{n}{k}&=\stirling{n-1}{k-1}+(n-1)\stirling{n-1}{k} \ \ \text{ if } 0<k\leq n
\end{align*} 
with the same base cases as $c(n,k)$. The sign is given by $c(n,k)=(-1)^{n+k}\stirling{n}{k}$.
\end{defi}
Stirling numbers are classical numbers in combinatorics. The unsigned version counts permutations on $n$ elements with exactly $k$ disjoint cycles. We deduce the following lemma, which we observed above for $n=2$ and which holds for all values of $n$.
\begin{lemma}
    \label{lem:derivativesoff}
    Let $n \in \NN_0, p \in \RR[y]_n$ and $f:x\mapsto \frac{p(\log(x))}{x} \in \F_{1,n}$. For any $k\in \NN_0$ there exists a polynomial $h_k(y)\in \RR[y]_n$ such that \begin{align*}
        \left(\frac{\dd}{\dd x}\right)^kf(x)=\frac{h_k(\log(x))}{x^{k+1}}.
    \end{align*}
    Furthermore these polynomials satisfy the recursive relation \begin{align*}
        h_k(y)=\left(\frac{\dd}{\dd y}\right)h_{k-1}(y)-kh_{k-1}(y)
    \end{align*}
    and therefore \begin{align*}
    h_k(y)=\sum_{l=0}^kc(k+1,l+1)\left(\frac{\dd}{\dd y}\right)^lp(y).
    \end{align*}
\end{lemma}
\begin{proof}
    This follows by induction on $k$.
\end{proof}
It turns out to be useful to express Stirling numbers in terms of harmonic numbers $H_n^{(m)}=\sum_{i=1}^n\frac{1}{i^m}$, which we introduce next. 

\begin{defi}
    For $k \in \NN_0$ we define a  polynomial $w(k)\in \ZZ[H_1,\ldots,H_k]$ in $k$ formal variables recursively by \begin{align*}
        w(0)&=1 \\
        w(k)&=\sum_{m=0}^{k-1}(-1)^m \frac{(k-1)!}{(k-m-1)!}H_{m+1}w(k-m-1).
    \end{align*}
\end{defi}
We mostly use that these polynomials can be obtained from an exponential generating function. 
\begin{lemma}
\label{lem:exp_generating_fct}
    For any $k \in \NN_0$ we have \begin{align}
    \label{eq_generatingfct}
        \frac{w(k)}{k!}=[x^k]\exp\left(\sum_{m=1}^\infty\frac{(-1)^{m-1}H_m}{m}x^m\right).
    \end{align}
\end{lemma}
The notation $[x^k]f(x)$ for a formal power series $f(x)$ equals the coefficient of $x^k$ in $f(x)$.
\begin{proof}
   This follows from \cite[Section 3, Example 1]{Flajolet} by using that Stirling numbers of the first and second kind form inverse matrices. A second way to see this is by using Bell-polynomials, which we recall in Definition \ref{def:bell} below. The Lemma then states that $w(k)=B_k(\mathbf{x})$ for $x_m=(-1)^{m-1}(m-1)!H_m, m\in \NN$ and follows immediately by the recursive property of Bell-polynomials, see Lemma \ref{lem:easy_properties_bell}, part (c).
\end{proof}
The connection between $w(k)$ and Stirling numbers is explained by the following classical result. For $n \in \NN_0, m \in \NN$ we write \begin{align*}
    H_n^{(m)}:=\sum_{i=1}^n\frac{1}{i^m}.
\end{align*}
\begin{prop}
\label{prop:w=stirling}
    For any $n,k \in \NN_0$ we have \begin{align*}
        \frac{1}{n!}c(n+1,k+1)=(-1)^{n+k}\left.\frac{w(k)}{k!}\right|_{H_i=H_{n}^{(i)}}.
    \end{align*}
\end{prop}
\begin{proof}
    See for example \cite[Section 2]{StirlingAndEulerSums} for the unsigned version. The signed version is obtained by simply multiplying both sides by $(-1)^{n+k}$.
\end{proof}

Returning to the case $n=2$ the above formula for $h_k(y)$ has three summands which we can write using Stirling numbers or using harmonic numbers.
\begin{align*}
h_k(y)&=c(k+1,1)(c_2 y^2+c_1 y+c_0)+c(k+1,2)(2c_2 y+c_1)+c(k+1,3)2c_2 \\
&=(-1)^kk!\left(c_2 y^2+\left(c_1 -2 c_2 H_k^{(1)}\right)y+c_0-c_1 H_k^{(1)}+c_2 \left(\left(H_k^{(1)}\right)^2-H_k^{(2)}\right)\right).
\end{align*}
The second equality follows from Proposition \ref{prop:w=stirling} by using $w(0)=1,w(1)=H_1,w(2)=H_1^2-H_2$.

Since the logarithm is surjective onto $\RR$, non-negativity of $\left(-\frac{\dd}{\dd x}\right)^kf(x)$ for all $x>0$ is equivalent to non-negativity of \begin{align*}
    g_k(y):=c_2 y^2+\left(c_1 -2c_2 H_k^{(1)}\right)y+c_0-c_1H_k^{(1)}+c_2\left(\left(H_k^{(1)}\right)^2-H_k^{(2)}\right).
\end{align*}
for all $y\in \RR$.
 If $c_2=0$ then non-negativity of $g_k(y)$ is equivalent to $c_0\geq 0,c_1=0$. Also for $c_2<0$ clearly $g_k(y)$ will become negative for large values of $y$, hence we assume from now on $c_2>0$.
 
 Non-negativity of a degree two polynomial can be read off from its discriminant.
This means that the polynomial $g_k(y)$ is non-negative whenever the following inequality holds \begin{align*}
    0\geq\disc(g_k)&=\left(c_1- 2 c_2 H_k^{(1)}\right)^2-4c_2 \left(c_0-c_1H_k^{(1)}+c_2\left(\left(H_k^{(1)}\right)^2-H_k^{(2)}\right)\right) \\
    &=c^2_1- 4c_0 c_2+ 4c_2^2 H_k^{(2)}.
\end{align*}
Interestingly, we see that the numbers $H_k^{(1)}$, which diverge as $k$ tends to infinity, cancel out. The above expression is monotonously increasing in $k$ and converges to \begin{align*}
   \disc(g_\infty):=c_1^2-4c_0c_2+4c_2^2\zeta(2)=c_1^2-4c_0c_2+\frac{2c_2^2\pi^2}{3}.
\end{align*}
This means that $g_k(y)$ is non-negative in $y$ for every $k$ if and only if \begin{align*}
    \disc(g_\infty)=c_1^2-4c_0c_2+\frac{2c_2^2\pi^2}{3}\leq 0.
\end{align*} 
We therefore reduced checking complete monotonicity within the space $\F_{1,2}$ to checking a single inequality among the coefficients $c_0,c_1,c_2$. The monotonicity of the discriminant in $k$ is visualized in  Figure \ref{fig:ShrinkCone2}.

\begin{figure}
    \centering
         \includegraphics[width=0.7\textwidth]{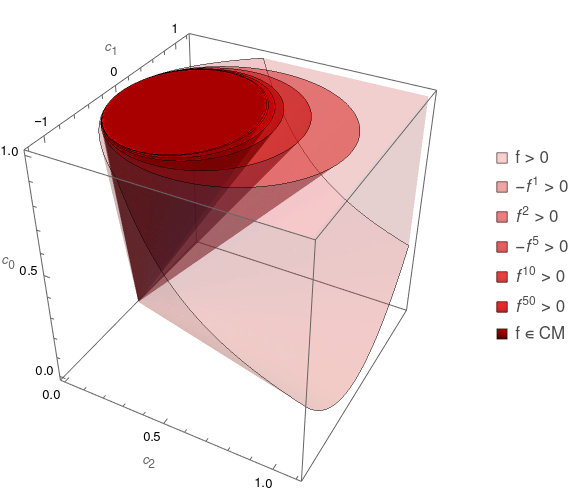}
         \caption{Outer approximations of the complete monotonicity cone in $c_0,c_1,c_2$ for the family of functions $f(x)=\frac{c_2 \log(x)^2+ c_1 \log(x)+c_0}{x}$. It shows the nested structure for the area of nonnegativity of signed derivatives of $f(x)$. The CM region in parameter space is the intersection of all of these regions.}
         \label{fig:ShrinkCone2}
\end{figure}
\begin{rem}
    The sequence of polynomials $g_k(y)$ does not converge coefficientwise for $k\to \infty$, yet the sequence of discriminants does converge. We can explain this by noting that the discriminant is invariant under a shift of the variable $y\to y+a, a \in \RR$. Indeed choosing $a=H_k^{(1)}$ we get \begin{align*}
\widetilde{g_k}:=g_k(y+H_k^{(1)})=c_2y^2+c_1y+c_0-c_2H_k^{(2)}.
    \end{align*}
    Under this shift we therefore see that the divergent part $H_k^{(1)}$ already formally cancels out. The sequence of polynomials $\widetilde{g_k}$ converges to \begin{align*}
        \widetilde{g_\infty}:=\lim_{k\to \infty}\widetilde{g_k}=c_2y^2+c_1y+c_0-c_2\frac{\pi^2}{6}.
    \end{align*}
    Indeed we see $\disc(g_\infty)=\disc(\widetilde{g_\infty})$.
\end{rem}

\begin{rem}
\label{rem:shrinking}
We find a phenomenon of contraction of the region of non-negativity of the discriminant of $\widetilde{g_k}$ in Figure \ref{fig:ShrinkCone2}, as $k$ grows. This is due to the negativity of the difference
\begin{equation*}
    \disc(g_{k-1})- \disc(g_k)=-4c_2^2/k^2<0
\end{equation*}
From this point of view, we can confirm that the CM region of $f(x)$ is exactly the non-positive region of $\disc(\widetilde{g_\infty})$.
\end{rem}

The polynomial $\widetilde{g_\infty}$, whose non-negativity certifies complete monotonicity, is also obtained in another way. Let us denote the \emph{Laplace transform} of a (non-negative) Borel measure $\alpha$ supported on $\RR_{\geq 0}^s$ by \begin{align*}
    \L(\alpha):\RR_{>0}^s\to \RR, x \mapsto \int_{\RR_{\geq0}^s} \exp\left(-\sum_{i=1}^sx_it_i\right)\dd\alpha(\mathbf{t}).
\end{align*}

The following important Theorem was proven by Bernstein, Hausdorff and Widder in the univariate case $s=1$ \cite[Theorem 12a]{Widder} and generalized to the multivariate setting by Choquet \cite[Theorem 10]{choquet}.
\begin{thm}[Bernstein-Hausdorff-Widder, Choquet]
\label{thm:HBWC}
    A function $f:\RR_{>0}^s\to \RR$ is completely monotone if and only if $f=\L(\alpha)$ for a non-negative Borel measure $\alpha$ supported on $\RR_{\geq 0}^s$.
\end{thm}

 In all of the cases studied in this article the measure $\alpha$ is absolutely continuous with respect to the Lebesque measure. This means that we find a continuous function $q$ on $\RR_{>0}^s$ such that $\dd \alpha(\mathbf{t})=q(\mathbf{t})\dd \mathbf{t}$ for $\dd t$ the Lebesque measure restricted to $\RR^s_{> 0}$. We will therefore consider the operator $\L$ as a map between continuous functions on the positive orthant, completely forgetting about measures. The theorem above then states that checking complete monotonicity comes down to checking non-negativity of $q$. In theory the function $q$  can be computed as the inverse Laplace transform $\L^{-1}(f)$. However, in practice computing the inverse Laplace transform can be very hard, especially when $f$ involves parameters like $c_0,c_1,c_2$.

For $n=2,s=1$ and $f(x)=\frac{c_2\log(x)^2+c_1\log(x)+c_0}{x}$, using the software Wolfram Mathematica \cite{Mathematica}, the inverse Laplace transform of $f$ is symbolically computable. The result is \begin{align*}
    \L^{-1}(f)(t)=c_2(\log(t)+\gamma)^2-c_1(\log(t)+\gamma)+c_0-c_2\frac{\pi^2}{6} =\widetilde{g_\infty}(-\log(t)-\gamma)
\end{align*}
where $\gamma\coloneqq\lim_{n\to \infty}(H_n^{(1)}-\log(n))$ is the Euler-Mascheroni constant.
Again we see that this function is non-negative on $\RR_{>0}$ if and only if $\widetilde{g_\infty}$ is a non-negative polynomial. This confirms our result from the first approach.

This simple example gives us a hint for identifying $\widetilde{g_\infty}(-\log(x)-\gamma)$ and $\L^{-1}(f)(x)$ for general $n\in \NN_0$, see Section \ref{sec:infinite_order_derivative}. While for the $n=2$ function family, we can describe non-negativity of $\widetilde{g_\infty}$ completely by discriminants, we will see that for higher $n$, the description of the CM region is more complicated.
\section{Positivity of the inverse Laplace transform}
\label{sec:main_result}
In the upcoming section we will generalize the second approach from the previous example to arbitrary $s,n$. To do so we need to compute the inverse Laplace transform $\L^{-1}(f)$ of any function $f\in \F_{s,n}$. We focus first on the univariate case $s=1$. Due to the increased complexity of the appearing integrals we were unable to use software in the same way as we did for the quadratic case for $n>2$. However, looking at the degree two case one can expect the operator $\L^{-1}$ to send $f$ to a polynomial in logarithms. This motivates the following definition and proposition. From now on we fix $n\in \NN_0$.
\begin{defi}
    Let $\G_{1,n}$ denote the $(n+1)$-dimensional subspace of $\C^\infty(\RR_{>0},\RR)$ which is given by \begin{align*}
        \G_{1,n}\coloneqq\{g\in \C^{\infty}(\RR_{>0},\RR)\mid \exists p\in \RR[y]_n: \forall x \in \RR_{>0}: g(x)=p(\log(x))\}.
    \end{align*}
    We fix two natural bases for $\F_{1,n}$ and $\G_{1,n}$ given by \begin{align*}
        \B_\F&\coloneqq\left(\left[x\mapsto \frac{\log(x)^j}{x}\right] \right)_{j=0,\ldots,n}\\
        \B_\G&\coloneqq\left(\left[x\mapsto \log(x)^j\right] \right)_{j=0,\ldots,n}.
    \end{align*}
\end{defi}
Recall the notation \begin{align*}
    g_l\coloneqq\left.\left(\frac{\dd }{\dd z}\right)^l\Gamma(z)\right|_{z=1}.
\end{align*}
\begin{prop}
\label{prop:Rourous_prop}
    The Laplace transform restricts to an isomorphism of finite dimensional $\RR$-vector spaces \begin{align*}
        \L|_{\G_{1,n}}:\G_{1,n}\to \F_{1,n}.
    \end{align*}
    The representing matrix $A=(A_{m,k})_{0\leq m,k\leq n}$ with respect to the two bases $\B_\G,\B_\F$ is triangular and is defined by \begin{align*}
        A_{m,k}=\begin{cases}
            (-1)^m\binom{k}{m} g_{k-m} &\text{if }k\geq m \\
            0 & \text{else}
        \end{cases}.
    \end{align*}
\end{prop}
\begin{proof}
    It suffices to show that $\L$ maps the function $\log^k:=[x\mapsto \log(x)^k]$ to the function \begin{align*}
        \left[x\mapsto \sum_{m=0}^kA_{m,k}\frac{\log(x)^m}{x}\right].
    \end{align*}
    Indeed, we compute for any $x \in \RR_{>0}$,
    \begin{align*}
        \L(\log^k)(x)&=\int_{0}^\infty \log(t)^k\exp(-tx)\dd t.
    \end{align*}
    We apply the substitution $z=tx$ to obtain \begin{align*}
        \L(\log^k)(x)&=\int_{0}^\infty\frac{\log\left(\frac{z}{x}\right)^k}{x}\exp(-z)\dd z \\
        &=\int_{0}^\infty\frac{(\log(z)-\log(x))^k}{x}\exp(-z)\dd z \\
        &= \sum_{m=0}^k(-1)^{m}\binom{k}{m}\frac{\log(x)^m}{x}\int_0^\infty\log(z)^{k-m}\exp(-z)\dd z.
    \end{align*}
    We therefore need to show for any $m\leq k$ that \begin{align*}
        \int_0^\infty \log(z)^{k-m}\exp(-z)\dd z=g_{k-m}.
    \end{align*}
    Indeed, we compute
    \begin{align*}
        g_{k-m} &=\left.\left(\frac{\dd}{\dd y}\right)^{k-m}\Gamma(y)\right|_{y=1}\\
        &= \left.\left(\frac{\dd}{\dd y}\right)^{k-m}\int_0^\infty z^{y-1}\exp(-z)\dd z\right|_{y=1}\\
        &=\left.\int_0^\infty\left(\frac{\dd}{\dd y}\right)^{k-m} \exp(\log(z^{y-1}))\exp(-z)\dd z\right|_{y=1}\\
        &=\left.\int_0^\infty\left(\frac{\dd}{\dd y}\right)^{k-m} \exp((y-1)\log(z))\exp(-z)\dd z\right|_{y=1} \\
        &= \left.\int_0^\infty\log(z)^{k-m} \exp((y-1)\log(z))\exp(-z)\dd z\right|_{y=1} \\
        &= \left.\int_0^\infty\log(z)^{k-m} z^{y-1}\exp(-z)\dd z\right|_{y=1} \\
        &=\int_0^\infty \log(z)^{k-m}\exp(-z) \dd z,
    \end{align*}
    where we differentiate under the integral sign.
\end{proof}
Combining Theorem \ref{thm:HBWC} with Proposition \ref{prop:Rourous_prop} we obtain Theorem \ref{thm:intro_main_thm} which we restate here.
\begin{thm}

        Let $p=\sum_{i=0}^nc_iy^i\in \RR[y]_n$ and let $f:x \mapsto \frac{p(\log(x))}{x}$ be the corresponding function in $\F_{1,n}$. Write $c=(c_0,\ldots,c_n)^T$ and $Y=(1, y,\ldots,y^n)^T$, and let $A\in \RR^{(n+1)\times (n+1)}$ be the matrix from Proposition \ref{prop:Rourous_prop}. Then $f$ is completely monotone if and only if \begin{align*}
            Y^TA^{-1}c\in \Sigma_{1,n}.
        \end{align*}
\end{thm}
\begin{proof}
    Substituting $y=\log(x)$ in $\sum_{k=0}^n(A^{-1}c)_ky^k$, by Proposition \ref{prop:Rourous_prop} gives $\L^{-1}(f)$. Since the logarithm is surjective, $\sum_{k=0}^n(A^{-1}c)_ky^k$ is non-negative for every $y$ if and only if $\L^{-1}(f)$ is nonnegative, which by Theorem \ref{thm:HBWC} happens precisely for completely monotone $f$. Since a univariate polynomial is non-negative if and only if it is a sum of squares, this concludes the proof.
\end{proof}

As an immediate corollary we obtain that the cone $\CM_{1,n}$ is a semi-algebraic set. This means that the set $\CM_{1,n}$ is the solution to finitely many polynomial inequalities in the coordinates with respect to the basis $\B_\F$.

\begin{cor}
\label{cor:semialg}
    The cone $\CM\cap \F_{1,n}\subseteq \F_{1,n}\cong \RR^{n+1}$ is semi-algebraic. 
\end{cor}
\begin{proof}
    The cone $\Sigma_{1,n}=\Psd_{1,n}$ is the complement of the projection of the semi-algebraic set \begin{align*}
        \left\{(\mathbf{c},y)\in \RR^{n+1}\times \RR \middle| \sum_{i=0}^nc_iy^i< 0\right\}
    \end{align*}
    onto the first $n+1$ factors. By the Tarski-Seidenberg-Theorem \cite[Proposition 5.2.1]{TSP}, this implies that $\Sigma_{1,n}$ is semi-algebraic. Since vector-space isomorphisms preserve this property, also $A\Sigma_{1,n}=\CM\cap \F_{1,n}$ is semi-algebraic.
\end{proof}

The matrix $A$ can easily be inverted due to its triangular shape. To write down explicit formulas for its entries we recall the notion of Bell-polynomials.

\begin{defi}
\label{def:bell}
    Let $\mathbf{x}=(x_n)_{n \in \NN}$ be a sequence. The \emph{$k$-th exponential complete Bell-polynomial} in $\mathbf{x}$ is defined as \begin{align*}
        B_k(\mathbf{x}):=k![t^k]\exp\left(\sum_{j=1}^\infty x_j\frac{t^j}{j!}\right).
    \end{align*}
\end{defi}
The entries of $\mathbf{x}$ can be viewed as formal variables, they just have to belong to any $\mathbb{Q}$-algebra for the definition to make sense. We refer to \cite[Section 3.3]{Comtet} for details on Bell-polynomials.\footnote{The reference uses partial Bell-polynomials, which result in our definition when summing over one of the two indices in \cite{Comtet}. All of the results that we use for ordinary Bell-polynomials follow directly from their respective versions for partial Bell-polynomials in Comtet's book.}
The following properties follow easily from this definition.
\begin{lemma}
\label{lem:easy_properties_bell}
    Let $\mathbf{x}=(x_n)_{n \in \NN}$ and $\mathbf{z}=(z_n)_{n \in \NN}$ be sequences and let $\mathbf{y}=(\alpha^nx_n)_{n\in \NN}$ for some $\alpha \in \RR$. \begin{enumerate}[label=(\alph*)]
        \item $B_k(\mathbf{y})=\alpha^kB_k(\mathbf{x})$
        \item $B_k(\mathbf{x}+\mathbf{z})=\sum_{m=0}^k\binom{k}{m}B_{k-m}(\mathbf{x})B_m(\mathbf{z})$
        \item  $B_n(\mathbf{x})=\sum_{j=0}^{n-1}\binom{n-1}{j}B_{n-1-j}(\mathbf{x})x_{j+1}$
    \end{enumerate}
\end{lemma}
The following well-known result allows us to compute higher derivatives of the gamma function around 1 in terms of zeta values.
\begin{prop}
\label{prop:gamma-Bell}
    The $k$-th derivative of $\Gamma(z)$ at $z=1$ is given by \begin{align*}
        g_k=\left.\left(\frac{\dd}{\dd z}\right)^k\Gamma(z)\right|_{z=1}= (-1)^kB_k(\gamma,1!\zeta(2),2!\zeta(3),3!\zeta(4),\ldots).
    \end{align*}
    More formally, the Bell-polynomial on the right is evaluated in the sequence $\mathbf{x}$ with \begin{align*}
        x_n=\begin{cases}
            \gamma &\text{ if } n=1 \\
            (n-1)!\zeta(n) &\text{ else}
        \end{cases}.
    \end{align*}
\end{prop}
\begin{proof}
    In \cite[Theorem 10.6.1]{boros2004expansion} the authors prove a series expansion of the logarithm of the gamma function around 1, namely \begin{align*}
        \log(\Gamma(1+z))=-\gamma z +\sum_{k=2}^\infty\frac{(-1)^k\zeta(k)}{k}z^k
    \end{align*}
    for $|z|<1$. Using the sequence $x_n$ we can rewrite this as \begin{align*}
\log(\Gamma(1+z))=\sum_{k=1}^\infty(-1)^kx_k\frac{z^k}{k!}
    \end{align*}
    and taking the exponential on both sides we obtain \begin{align*}
\Gamma(1+z)=\exp\left(\sum_{k=1}^\infty(-1)^kx_k\frac{z^k}{k!}\right).
    \end{align*}
    Taking the $k$-th derivative and evaluating at $z=0$ up to the factor $k!$ simply extracts the $z^k$ coefficient in the series on the right, hence \begin{align*}
      \left.\left(\frac{\dd}{\dd z}\right)^k\Gamma(z)\right|_{z=1}=\left.\left(\frac{\dd}{\dd z}\right)^k\Gamma(1+z)\right|_{z=0}= k![z^k]\exp\left(\sum_{k=1}^\infty(-1)^kx_k\frac{z^k}{k!}\right).
    \end{align*}
    By definition the right hand side is $B_k(\mathbf{y})$ where $y_n=(-1)^nx_n,n\in \NN$. Using part (a) of Lemma \ref{lem:easy_properties_bell} this is equal to $(-1)^kB_k(\mathbf{x})$ as claimed.
\end{proof}

With these tools we can write down the inverse matrix of $A$ explicitly.

\begin{lemma}
\label{lem:Ainverse}
    Define a matrix $C=(C_{m,k})_{0\leq m,k \leq n}$ via \begin{align*}
        C_{m,k}:=\begin{cases}
            (-1)^m\binom{k}{m}B_{k-m}(-\gamma,-1!\zeta(2),-2!\zeta(3),-3!\zeta(4),\ldots) & \text{ if }k\geq m \\
            0 & \text{else}
        \end{cases}.
    \end{align*}
    Then $CA=AC=\id_{\RR^{n+1}}$ and hence $C=A^{-1}$.
\end{lemma}
\begin{proof}
    First let us write the matrix $A$ using Bell-polynomials. By Proposition \ref{prop:gamma-Bell} we have for all indices $0\leq m,k\leq n$\begin{align*}
        A_{m,k}=\begin{cases}
            (-1)^k\binom{k}{m}B_{k-m}(\gamma,1!\zeta(2),2!\zeta(3),3!\zeta(4),\ldots) & \text{ if } k\geq m \\
            0 &\text{ else}
        \end{cases}.
    \end{align*}
    It suffices to show $AC=\id_{\RR^{n+1}}$ which we check for every pair of indices $m,k$. Both $A$ and $C$ are clearly triangular matrices with the diagonal $A_{m,m}=C_{m,m}=(-1)^m$ Hence their product is triangular with 1's on the diagonal. We may therefore assume $m<k$, now wishing to show $(AC)_{m,k}=0$. Set $\mathbf{x}=(\gamma,1!\zeta(2),2!\zeta(3),3!\zeta(4),\ldots)$. Then we compute \begin{align*}
(AC)_{m,k}&=\sum_{j=0}^nA_{m,j}C_{j,k} \\&= \sum_{j=m}^k (-1)^j\binom{j}{m}B_{j-m}(\mathbf{x})(-1)^j\binom{k}{j}B_{k-j}(-\mathbf{x}) \\
&=
\sum_{j=0}^{k-m}\binom{j+m}{m}\binom{k}{j+m}B_j(\mathbf{x})B_{k-m-j}(-\mathbf{x}) \\&=\sum_{j=0}^{k-m}\binom{k}{m}\binom{k-m}{j}B_j(\mathbf{x})B_{k-m-j}(-\mathbf{x}) \\&=
\binom{k}{m}B_{k-m}(\mathbf{x}-\mathbf{x}) \\&=0.
    \end{align*}
Here we used Lemma \ref{lem:easy_properties_bell} and the obvious fact that $B_l(\mathbf{0})=0$ for all $l>0$.
\end{proof}

We now generalize Theorem \ref{thm:intro_main_thm} to the multivariate case $s>1$. Recall that\begin{align*}
    T^{s,n}\varphi_A: \RR[{\bf y}]_n\to \F_{s,n}
\end{align*}  
is the linear map sending the monomial $y_1^{\lambda_1}\cdots y_s^{\lambda_s}, \sum_{i=1}^s\lambda_i\leq n$, to the function \begin{align}
\label{eq:defT}
(x_1,\ldots,x_s)\mapsto \prod_{j=1}^s\varphi_A(y_j^{\lambda_j})(x_j)\in \F_{s,n}.
\end{align}
Here $\varphi_A$ is the operator $\RR[y]_n\to \F_{1,n}$ represented by the matrix $A$ in the monomial basis and in the basis $\B_\F$.
As for the univariate case we define 
\begin{align*}
        \G_{s,n}\coloneqq\{g\in \C^{\infty}(\RR_{>0}^s,\RR)\mid \exists p\in \RR[{\bf y}]_n: \forall \mathbf{x} \in \RR_{>0}^s: g(\mathbf{x})=p(\log(x_1),\ldots,\log(x_s))\}.
\end{align*}
We identify $\G_{s,n}$ with $\RR[{\bf y}]_n$ via the natural isomorphism $\iota:\RR[{\bf y}]_n\to \G_{s,n}$ which maps a polynomial $p$ to the function $\mathbf{x}\mapsto p(\log(x_1),\ldots,\log(x_s))$.
\begin{prop}
\label{prop:Rourous_prop_multi}
    The operator $T^{s,n}\varphi_A$ agrees with the multivariate Laplace transform restricted to $\G_{s,n}$ under the isomorphism $\iota$. In other words we have \begin{align*}
\L|_{\G_{s,n}}=T^{s,n}\varphi_A\circ\iota^{-1}
    \end{align*}
    as linear maps from $\G_{s,n}$ to $\F_{s,n}$.
\end{prop}
\begin{proof}
    Let $y_1^{\lambda_1}\cdots y_s^{\lambda_s}, \sum_{i=1}^s\lambda_i\leq n$ be a monomial in $\RR[{\bf y}]_n$. As such monomials form a basis of $\RR[{\bf y}]_n$ it suffices to show that the function (\ref{eq:defT}) is the Laplace transform of the function \begin{align*}
        \log^{\lambda}: \mathbf{x}\mapsto \prod_{j=1}^s\log(x_j)^{\lambda_j}.
    \end{align*}
    Using the univariate case (Proposition \ref{prop:Rourous_prop}) we indeed have \begin{align*}
        \L(\log^{\lambda})(\mathbf{x})&= \int_{\RR_{>0}^s}\exp\left(-\sum_{j=1}^sx_jt_j\right)\log^{\lambda}(\mathbf{t})\dd \mathbf{t}\\
        &=\int_{\RR_{>0}^s}\prod_{j=1}^s\exp\left(-x_jt_j\right)\log(t_j)^{\lambda_j}\dd \mathbf{t} \\
        &=\prod_{j=1}^s\int_{\RR_>0}\exp\left(-x_jt_j\right)\log(t_j)^{\lambda_j}\dd t_j \\
        &=\prod_{j=1}^s\varphi_A(y_j^{\lambda_j})(x_j).
    \end{align*}
\end{proof}
With the above proposition we can prove Theorem \ref{thm:multi_main_thm_intro}.
\begin{thm}
    Let $f\in \F_{s,n}$. Then $f$ is completely monotone if and only if $(T^{s,n}\varphi_A)^{-1}(f)\in \Psd_{s,n}$. In particular, $\CM_{s,n}$ is a semi-algebraic cone. 
\end{thm}
\begin{proof}
    First notice that $T^{s,n}\varphi_A$ is actually an isomorphism and hence invertible. This follows immediately from the fact that $\varphi_A$ is an isomorphism in the univariate setting. By the Bernstein-Hausdorff-Widder-Choquet Theorem \ref{thm:HBWC} we get that $f$ is completely monotone precisely when $\L^{-1}(f)$ is non-negative on the positive orthant. Since the logarithm maps the positive halfline surjectively onto $\RR$ this happens if and only if $\iota^{-1}\circ \L^{-1}(f)\in \Psd_{s,n}$ is a non-negative polynomial. By Proposition \ref{prop:Rourous_prop_multi} we have $\iota^{-1}\circ \L^{-1}(f)=(T^{s,n}\varphi_A)^{-1}(f)$ which proves the first part. Similar to the univariate case (Corollary \ref{cor:semialg}) the Tarski-Seidenberg Theorem implies that $\CM_{s,n}$ is a semi-algebraic cone.
\end{proof}

We close this section by discussing the degree two case $n=2$ for arbitrary $s$.
\begin{ex}
    Let us consider $f\in \F_{s,2}$. Such a function can be represented by a unique symmetric matrix $D$ satisfying \begin{align*}
        f(x)=\frac{1}{x_1\ldots x_s}(\log(x_1),\ldots, \log(x_s),1)D(\log(x_1),\ldots, \log(x_s),1)^T.
    \end{align*}
    From Section \ref{sec:example_n=2} we know the univariate Laplace transforms satisfy \begin{align*}
        \L^{-1}\left(\frac{1}{x_i}\right)=1, \L^{-1}\left(\frac{\log(x_i)}{x_i}\right)=-\log(x_i)-\gamma, \L^{-1}\left(\frac{\log(x_i)^2}{x_i}\right)=(-\log(x_i)-\gamma)^2-\frac{\pi^2}{6}.
    \end{align*}
    By replacing the surjective function $-\log(x_i)-\gamma$ by a new variable $y_i$ and employing Proposition \ref{prop:Rourous_prop_multi} we hence get that \begin{align*}
        \L^{-1}(f)=(y_1,\ldots,y_s,1)\widetilde{D}(y_1,\ldots,y_s,1)^T,
    \end{align*}
    where the matrix $\widetilde{D}$ is obtained from $D$ by adding $-\frac{\pi^2}{6}\sum_{i=1}^sD_{i,i}$ to the bottom right entry of $D$. \begin{align*}
        \widetilde{D}=D+\begin{bmatrix}
0 & \cdots & 0 & 0 \\
\vdots &  & \vdots & \vdots \\
0 & \cdots & 0 & 0 \\
0 & \cdots & 0 & -\frac{\pi^2}{6}\sum_{i=1}^sD_{i,i}
\end{bmatrix} .
    \end{align*}
    The function $\L^{-1}(f)$ is non-negative precisely when the matrix $\widetilde{D}$ is positive semi-definite. Therefore the cone $\CM_{s,2}$ is the semi-algebraic set described by non-negativity of all principal minors of $\widetilde{D}$.
\end{ex}
\section{Infinite order derivatives}
\label{sec:infinite_order_derivative}
In the upcoming section we focus on the univariate setting $s=1$.
For the degree two case, in Remark \ref{rem:shrinking} we observed a convergence behavior in the regions of $\F_{1,2}$ that are cut out by the conditions $\left(\frac{-\dd}{\dd x}\right)^kf(x)\geq 0$ as $k$ tends to infinity. In fact these regions were nested, forming a descending chain with limit precisely $\CM_{1,2}$. While we can not prove the same containment and convergence for every degree $n$, we can still see that the polynomial conditions $\widetilde{g_k}\geq0$, which are equivalent to nonnegativity of the signed derivative $\left(\frac{-\dd}{\dd x}\right)^kf(x)$, do converge to a condition $\widetilde{g_\infty}\geq0$ which is equivalent to complete monotonicity of $f$.

As above we fix $n \in \NN$ and $p:=\sum_{k=0}^nc_ky^k\in \RR[c_0,\ldots,c_n][y]$. We write $f:\RR_{>0}\to \RR, x\mapsto \frac{p(\log(x))}{x}$ and $h_m(y)=\sum_{k=0}^mc(m+1,k+1)p^{(k)}(y)$. Recall from Lemma \ref{lem:derivativesoff} that  \begin{align*}
    f^{(m)}(x)=\frac{h_m(\log(x))}{x^{m+1}}.
\end{align*}
Let $g(y)\in \RR[c_0,\ldots,c_n,H_1,\ldots,H_n][y]$ be defined as \begin{align*}
    g(y)=\sum_{k=0}^n(-1)^k\frac{w(k)}{k!}p^{(k)}(y).
\end{align*}
Then by Proposition \ref{prop:w=stirling} we have \begin{align*}
    (-1)^nn!g(y)|_{H_i=H_n^{(i)}}=h_n(y).
\end{align*}
We further set \begin{align*}
    \widetilde{g}(y):=g(y+H_1).
\end{align*}
Notice that we clearly have \begin{align*}
    f \text{ is CM }&\Leftrightarrow \forall y\in \RR:\forall n\in \NN:(-1)^nh_n(y)\geq 0 \\&\Leftrightarrow \forall y\in \RR:\forall n \in \NN: g(y)|_{H_i=H^{(i)}_n}\geq 0 \\&\Leftrightarrow \forall y\in \RR:\forall n \in \NN: \widetilde{g}(y)|_{H_i=H^{(i)}_n}\geq 0.
\end{align*}
We write $g_m(y),\widetilde{g_m}(y)$ for the univariate polynomials obtained from $g(y),\widetilde{g}(y)$ by substituting $H_i=H^{(i)}_m$. 

Our next result states that $\widetilde{g}(y)$ is in fact independent of $H_1$.
\begin{prop}
\label{prop:H_1-independence}
    We have \begin{align*}
        \frac{\partial}{\partial H_1}\widetilde{g}(y) =0.
    \end{align*}
    as polynomials in $y$.
\end{prop}

To prove this we need to understand the derivatives in $H_1$ of $w(k)$. 

\begin{lemma}
\label{lem:derivativeH1}
    For any $k \in \NN_0$ we have \begin{align*}
        \frac{\dd}{\dd H_1}w(k+1)= (k+1)w(k)
    \end{align*}
\end{lemma}
\begin{proof}
    This follows easily from Lemma \ref{lem:exp_generating_fct}.
\end{proof}

\begin{proof}[Proof of Proposition \ref{prop:H_1-independence}]
    Plugging in the definition of $\widetilde{g}$ and using Lemma \ref{lem:derivativeH1} we write \begin{align*}
        \frac{\partial}{\partial H_1}\widetilde{g}(y)&=\frac{\partial}{\partial H_1}\sum_{k=0}^n(-1)^k\frac{w(k)}{k!}p^{(k)}(y+H_1) \\
        &=\sum_{k=1}^n (-1)^k \frac{kw(k-1)}{k!}p^{(k)}(y+H_1)+\sum_{k=0}^n(-1)^k\frac{w(k)}{k!}p^{(k+1)}(y+H_1) \\
        &=-\sum_{k=0}^{n-1} (-1)^k \frac{w(k)}{k!}p^{(k+1)}(y+H_1)+\sum_{k=0}^n(-1)^k\frac{w(k)}{k!}p^{(k+1)}(y+H_1) \\
        &=(-1)^{n}\frac{w(n)}{n!}p^{(n+1)}(y+H_1) \\&=0.
    \end{align*}
\end{proof}
Since all of the sequences $(H_m^{(i)})_{m \in \NN}$ for $i > 1$ converge to zeta values $\zeta(i)$ we may now define \begin{align*}
    \widetilde{g_\infty}(y):=\lim_{m\to \infty}\widetilde{g}(y)|_{H_i=H_m^{(i)}}=\widetilde{g}(y)|_{H_i=\zeta(i)}.
\end{align*}
Loosely speaking non-negativity of $\widetilde{g_\infty}(y)$ as a polynomial in $y$ certifies that the derivative of $f$ ``to infinite order'' has the correct sign.

Writing down a closed formula for $\widetilde{g_\infty}(y)$ at this point seems hard, since in the given formulas all appearances of $H_1$ cancel out and only after the cancellation we can set $H_i=\zeta(i)$. The next proposition gives a closed formula for $\widetilde{g}$ which is evidently independent of $H_1$. In particular this gives a closed formula for $\widetilde{g_\infty}(y)$ simply by substituting $H_i=\zeta(i)$.
\begin{prop}
\label{prop:tildeg_formula}
    We have \begin{align*}
        \widetilde{g}(y)=\sum_{k=0}^nc_k(-1)^kw(k)|_{H_1=-y}
    \end{align*}
    as polynomials in $y,c_0,\ldots,c_n,H_2,\ldots,H_n$. In particular \begin{align*}
        \widetilde{g_\infty}(y)=\sum_{k=0}^nc_k(-1)^kw(k)|_{H_1=-y,H_i=\zeta(i)}.
    \end{align*}
\end{prop}
\begin{proof}
    Since for any $0\leq j\leq n$ we have \begin{align*}
        p^{(j)}(y+H_1)&=\sum_{k=j}^nc_k\frac{k!}{(k-j)!}(y+H_1)^{k-j} \\
        &=\sum_{k=j}^n\sum_{i=0}^{k-j}c_k\frac{k!}{(k-j)!}\binom{k-j}{i}y^iH_1^{k-j-i} \\
        &=\sum_{k=j}^n\sum_{i=0}^{k-j}j!\binom{k}{j}\binom{k-j}{i}c_ky^iH_1^{k-j-i}.
    \end{align*}
    Substituting this in the definition of $\widetilde{g}(y)$ we get \begin{align*}
        \widetilde{g}(y)&=\sum_{j=0}^n(-1)^j\frac{w(j)}{j!}p^{(j)}(y+H_1)\\
        &=\sum_{j=0}^n\sum_{k=j}^n\sum_{i=0}^{k-j}(-1)^j\binom{k}{j}\binom{k-j}{i}w(j)c_ky^iH_1^{k-j-i} \\&=\sum_{k=0}^n\sum_{j=0}^{k}\sum_{i=0}^{k-j}(-1)^j\binom{k}{j}\binom{k-j}{i}w(j)c_ky^iH_1^{k-j-i}.
    \end{align*}
    Since we know from Proposition \ref{prop:H_1-independence} that $\widetilde{g}(y)$ does not depend on $H_1$, proving the claim of this proposition is equivalent to proving that \begin{align*}
        \widetilde{g}(-H_1)=\sum_{k=0}^nc_k(-1)^kw(k).
    \end{align*}
    Therefore setting $y=-H_1$ in the above we get \begin{align*}
        \widetilde{g}(-H_1)=\sum_{k=0}^n\sum_{j=0}^{k}\sum_{i=0}^{k-j}(-1)^{j+i}\binom{k}{j}\binom{k-j}{i}w(j)c_kH_1^{k-j}.
    \end{align*}
    The coefficient of the linear monomial $c_k$ for fixed $k\in \{0,\ldots,n\}$ is therefore equal to \begin{align*}
        [c_k]\widetilde{g}(-H_1)&=\sum_{j=0}^{k}\sum_{i=0}^{k-j}(-1)^{j+i}\binom{k}{j}\binom{k-j}{i}w(j)H_1^{k-j}\\
        &=\sum_{j=0}^k(-1)^{j}\binom{k}{j}\left(\sum_{i=0}^{k-j}(-1)^i\binom{k-j}{i}\right)w(j)H_1^{k-j}.
    \end{align*}
    For any $m\in \NN_0$ we have \begin{align*}
        \sum_{i=0}^m(-1)^i\binom{m}{i}=(1-1)^m=\begin{cases}
            1 &\text{if }m=0 \\
            0& \text{else}.
        \end{cases}
    \end{align*}
    And therefore in the above expression for $[c_k]\widetilde{g}(-H_1)$ only one summand is non-zero, namely the summand for $j=k$. This summand equals \begin{align*}
        [c_k]\widetilde{g}(-H_1)=(-1)^kw(k)H_1^0=(-1)^{k}w(k)
    \end{align*}
    as claimed.
\end{proof}

While we can not show that non-negativity of $\widetilde{g_m}(y)$ implies already non-negativity of $\widetilde{g_{m-1}}(y)$ for all $m$, it is true that non-negativity of $\widetilde{g_\infty}$ implies non-negativity of every $\widetilde{g_m}(y)$, i.e. complete monotonicity of $f$. To see this we conclude this section by showing that, up to a change of variables, $\widetilde{g_\infty}$ agrees with the inverse Laplace transform of $f$.
\begin{prop}
For $f,\widetilde{g_\infty}$ defined as above we have
    \begin{align*}
        \widetilde{g_{\infty}}(-\log(x)-\gamma)=\L^{-1}(f)(x)
    \end{align*}
    for all $x \in \RR_{>0}$.
\end{prop}
\begin{proof}
    Both sides are linear in the parameters $c_k$ and polynomial in $\log(x)$. By Theorem \ref{thm:intro_main_thm} and by Lemma \ref{lem:Ainverse} the coefficient of $c_k\log(x)^m$ on the right-hand side is \begin{align*}
        A^{-1}_{m,k}=\begin{cases}
            (-1)^m \binom{k}{m}B_{k-m}(-\gamma,-1!\zeta(2),-2!\zeta(3),-3!\zeta(4),\ldots) &\text{ if }k\geq m \\
            0 &\text{ else}
        \end{cases}.
    \end{align*}
    To ease notation we write $\mathbf{d}=(-\gamma,-1!\zeta(2),-2!\zeta(3),-3!\zeta(4),\ldots)$. 

    By Proposition \ref{prop:tildeg_formula} the left-hand side of the claimed formula is \begin{align*}
        \widetilde{g_{\infty}}(-\log(x)-\gamma)= \left.\sum_{k=0}^nc_k(-1)^kw(k)\right|_{\substack{H_1=\log(x)+\gamma \\ H_i= \zeta(i)\ \ \ \ \ \ }}.
    \end{align*}
    We fix $k$ and look only at the $c_k$ coefficient which by Lemma \ref{lem:exp_generating_fct} is given by \begin{align}
    \label{eq:ugly_w_sub}
    \nonumber
        \left.(-1)^kw(k)\right|_{\substack{H_1=\log(x)+\gamma \\ H_i= \zeta(i)\ \ \ \ \ \ }} &= (-1)^kk![t^k]\exp\left((\log(x)+\gamma)t+\sum_{m=2}^\infty\frac{(-1)^{m-1}\zeta(m)}{m}t^m\right)\\
        &=\left.(-1)^k\left(\frac{\partial}{\partial t}\right)^k\exp\left((\log(x)+\gamma)t+\sum_{m=2}^\infty\frac{(-1)^{m-1}\zeta(m)}{m}t^m\right)\right|_{t=0}.
    \end{align}
    The term inside the exponential function produces almost precisely the sequence $d$ when taking higher derivatives and evaluating at $t=0$: \begin{align*}
        \left.\left(\frac{\partial}{\partial t}\right)^k\left((\log(x)+\gamma)t+\sum_{m=2}^\infty\frac{(-1)^{m-1}\zeta(m)}{m}t^m\right)\right|_{t=0}= \begin{cases}
            -d_1+\log(x) &\text{if }k=1 \\
            (-1)^kd_k &\text{else}
        \end{cases}.
    \end{align*}
    We write $\mathbf{d}'$ for this sequence.
    Now by Fa\`a di Bruno's formula, see  \cite[Section 3.4, Theorem A]{Comtet}, we get that (\ref{eq:ugly_w_sub}) in fact equals \begin{align*}
        (-1)^kB_k(\mathbf{d}')&= \sum_{m=0}^k\binom{k}{m}B_{k-m}(\mathbf{d})B_m(-\log(x),0,0,\ldots) \\
        &=\sum_{m=0}^k(-1)^m\binom{k}{m}B_{k-m}(\mathbf{d})\log(x)^m.
    \end{align*}
    Here the first equality follows by applying both parts of Lemma \ref{lem:easy_properties_bell} and the second equality is obvious from the definition of Bell-polynomials which we gave. In particular we see that the $c_k\log(x)^m$ coefficient of $\widetilde{g_\infty}(-\log(x)-\gamma)$ indeed equals $A^{-1}_{m,k}$ as claimed.
\end{proof}
\begin{figure}
    \centering
    \includegraphics[width=0.5\linewidth]{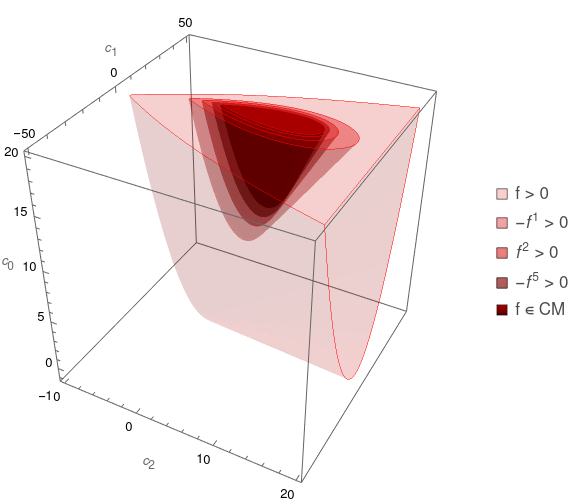}
    \caption{Outer approximations of the complete monotonicity cone in $c_0,c_1,c_2$ for the family of functions $f(x)=\frac{\log(x)^4+c_2 \log(x)^2+ c_1 \log(x)+c_0}{x}$. There seems be the similar nested structure as in Fig.\ref{fig:ShrinkCone2} which seems to converge to the CM region of $f(x)$ visually, see Conjecture \ref{conj:shrinking}.}
    \label{fig:ShrinkCone4}
\end{figure}
\begin{ex}
    Let us consider the case $n=4$. In this case there are five parameters $c_0,\ldots,c_4$ that determine a function $f\in \F_{1,4}$. Let us restrict ourselves to normalized degree four polynomials without a cubic term in the numerator, so that we can plot the resulting regions.\begin{align*}
        \F_{1,4}':=\left\{f\in \F_{1,4} \middle| \exists c_0,c_1,c_2\in \RR: \forall x\in \RR_{>0}: f(x)=\frac{\log(x)^4+c_2\log(x)^2+c_1\log(x)+c_0}{x}\right\}.
    \end{align*}
    The region of completely monotone functions in $\F_{1,4}'$ is given by the condition $\forall y\in \RR: \widetilde{g_\infty}(y)\geq 0$. We proved that this region is the intersection over all $m\in \NN_0$ of the regions cut out by the conditions \begin{align*}
       \forall x\in \RR_{>0}: \left(\frac{-\dd}{\dd x}\right)^mf(x)\geq 0.
    \end{align*}
    Plotting a few of these regions we again see that they even seem to form a chain under inclusion, see Figure \ref{fig:ShrinkCone4}.  We expect this nesting behavior to happen for all values of $n$, which leads us to formulate the following conjecture..
\end{ex}
\begin{conj}
\label{conj:shrinking}
    For any $n\in \NN$ the semialgebraic sets \begin{align*}
        D_k=\left\{f\in \mathcal{F}_{1,n}\middle|\forall x\in \RR:\left(\frac{-\dd}{\dd x}\right)^kf(x)\geq 0\right\}
    \end{align*}
    form a descending chain under inclusion, i.e. $D_k\supseteq D_{k+1}$, which converges to the region $\CM_{1,n}=\bigcap_{k=1}^\infty D_k$.
\end{conj}
\section{Acknowledgments}
This project started at a workshop hosted by the Max Planck Institute for Physics  in Garching, Germany, which was part of the ERC UNIVERSE+ Synergy Project (ERC, UNIVERSE PLUS, 101118787). We thank Johannes Henn, Prashanth Raman,  Bernd Sturmfels and Rainer Sinn for helpful discussion. We also thank the reviewers for their valuable feedback. RM is supported by the Outstanding PhD Students Overseas Study Support Program of University of Science and Technology of China. JW is supported by the SPP 2458 “Combinatorial Synergies”, funded
by the Deutsche Forschungsgemeinschaft (DFG, German Research Foundation),
project ID: 539677510.

\bibliography{bibML}
\bibliographystyle{plain} 
\end{document}